\documentclass[reqno]{amsart}
\usepackage[T1]{fontenc}
\usepackage[latin1]{inputenc}
\usepackage{amssymb}
\usepackage{amsfonts}
\usepackage{amsmath,mathtools}
\usepackage{mathrsfs}
\usepackage{graphicx}
\usepackage[notref, notcite]{showkeys}
\usepackage{color}
\usepackage{hyperref}
\usepackage{esint}
\usepackage{verbatim}
\usepackage{bbm}
\usepackage{tikz}
\usepackage{setspace}
\usepackage{stackrel}

\newcommand{\R}[1]{\mathbb{R}^{#1}}
\renewcommand{\S}[1]{\mathbb{S}^{#1}}

\newcommand{\cC}{\mathcal C}

\newcommand{\cF}{\mathcal F}

\newcommand{\cH}{\mathcal H}

\newcommand{\cL}{\mathcal L}
\newcommand{\cM}{\mathcal M}
\newcommand{\cO}{\mathcal O}

\newcommand{\bulk}{\mathrm{bulk}}

\newcommand{\surface}{\mathrm{surf}}

\newcommand{\Rb}         {\mathbb{R}}

\renewcommand{\geq}{\geqslant}
\renewcommand{\leq}{\leqslant}

\newcommand{\longrightharpoonup}{\relbar\joinrel\rightharpoonup}

\newcommand{\wsto}{\stackrel{*}{\rightharpoonup}}

\newcommand{\weakst}{\stackrel{*}{\longrightharpoonup}}
\newcommand{\weakstH}{\stackrel[H]{*}{\longrightharpoonup}}
\newcommand{\pweak}{\stackrel{p}{\longrightharpoonup}}

\newcommand{\average}{{\mathchoice {\kern1ex\vcenter{\hrule
height.4pt width 8pt depth0pt}
\kern-11pt} {\kern1ex\vcenter{\hrule height.4pt width 4.3pt
depth0pt} \kern-7pt} {} {} }}

\newcommand{\res}{\mathop{\hbox{\vrule height 7pt width .5pt depth
0pt\vrule height .5pt width 6pt depth0pt}}\nolimits}

\mathchardef\emptyset="001F

\providecommand{\U}[1]{\protect\rule{.1in}{.1in}}
\numberwithin{equation}{section}
\newtheorem{definition}{Definition}[section]
\newtheorem{theorem}[definition]{Theorem}

\newtheorem{corollary}[definition]{Corollary}

\theoremstyle{definition} {\newtheorem{remark}[definition]{Remark}}

\def\debaixodaseta#1#2{\mathrel{}\mathop{\longrightarrow}\limits^{#1}_{#2}}

\makeindex

\title[Relaxation for multi-level structured deformations]{Multi-level structured deformations: relaxation via a global method approach}%
\author[A.~C.~Barroso]{Ana Cristina Barroso}
\address[A.~C.~Barroso]{Departamento de Matem\'atica and CMAFcIO, 
Faculdade de Ci\^encias da Universidade de Lisboa,
Campo Grande, Edif\' \i cio C6, Piso 1,
1749-016 Lisboa, Portugal}
\email{acbarroso@ciencias.ulisboa.pt}
\author[J.~Matias]{Jos\'{e} Matias}
\address[J.~Matias]{Centro de Análise Matemática, Geometria e Sistemas Dinâmicos, Departamento de Matem\'atica, Instituto Superior T\'ecnico, Universidade de Lisboa, Av.~Rovisco Pais 1, 1049-001 Lisboa, Portugal}
\email{jose.c.matias@tecnico.ulisboa.pt}
\author[E.~Zappale]{Elvira Zappale}
\address[E.~Zappale]{Dipartimento di Scienze di Base ed Applicate per l'Ingegneria, Sapienza - Universit\`{a} di Roma, Via Antonio Scarpa, 16, 00161 Roma, Italy and
CIMA, Universidade de \'Evora, Portugal}
\email{elvira.zappale@uniroma1.it}

\date{\today}

	

\makeindex

\begin{document}

\begin{abstract}
We present some relaxation and integral representation results
for energy functionals in the setting of structured deformations, with special emphasis given to the case of multi-level structured deformations. In particular, we 
present an integral representation result for an abstract class of variational functionals in this framework via a global method for relaxation and identify,
under quite general assumptions, the corresponding relaxed energy densities through the study of a related local Dirichlet-type problem.

Some applications to specific relaxation problems are also mentioned, showing that our global method approach recovers some previously established results.
\end{abstract}

\maketitle

\allowdisplaybreaks

\tableofcontents


\section{Introduction}

The purpose of this work is to give an overview of some relaxation and integral representation results in the context of structured deformations.
The main focus is to present, and give an improvement, of some recent findings, obtained in \cite{BMZ2024}, where we established a global method for relaxation that is applicable in the context of multi-level structured deformations.
We also provide several applications that show that our approach covers, and in some
cases actually improves, in a unified way, various results available in the literature in the setting of (first-order) structured deformations.

The aim of the theory of (first-order) structured deformations 
(see \cite{DPO1993}) is to describe the effects, at the macroscopic level, of both smooth and non-smooth geometrical changes that a body may undergo and that occur at one sub-macroscopic level. A further step in this theory was undertaken in \cite{DO2019}, where the notion of a hierarchical system of structured deformations was introduced in order to include the effects, at the macro-level, of geometrical changes occurring at more than one sub-macroscopic level. We refer to Sections \ref{SDdef} and \ref{mlsd} for more details.

Since the seminal papers of Del Piero \& Owen \cite{DPO1993}, where structured deformations were originally introduced, and of Choksi \& Fonseca \cite{CF1997}, where the variational formulation for (first-order) structured deformations
in the $SBV$ setting was first addressed,
the theory has known many generalisations and extensions. 

For example, the notion of second-order structured deformations was introduced by Owen \& Paroni \cite{OP2000}, in order 
to describe curvature and bending effects associated with jumps in gradients.
Lower semi-continuity, relaxation and integral representation results in several functional spaces, such as $BH$, $SBH$ and $SBV^2$, and with respect to different topologies, have been obtained in 
\cite{ BMS2012, BMMO2017, FHP, H, S2015, S17}, for example.

Further contributions include applications to dimension reduction problems
as in \cite{CMMO} and to homogenisation \cite{AMMZ}. A recent article by
Kr\"omer, Kru\v zík, Morandotti \& Zappale \cite{KKMZ} extends the notion of structured deformations to so-called measure structured deformations.
We also refer to the survey by Matias, Morandotti \& Owen
\cite{book} where an overview of many of these topics may be found.

The global method for relaxation that we proposed in \cite{BMZ2024} allows us to provide an integral representation for a class of functionals defined in the set of $(L+1)$-level (first-order) structured deformations, 
$(g, G_1, \ldots, G_{L}) \in HSD_L^p(\Omega)$, where, for $L \in \mathbb N$ and 
$p \geq 1$, 
\begin{equation}\label{hsd}
HSD_L^p(\Omega):= SBV(\Omega;\mathbb R^{d})\times \underbrace{L^p(\Omega;\mathbb R^{d\times N})\times\cdots\times L^p(\Omega;\mathbb R^{d\times N})}_{L\text{-times}},
\end{equation}
and to identify the corresponding relaxed energy densities, under a set of assumptions that is quite general.
This is achieved through the study of a related local Dirichlet-type problem,
following the ideas of the global method for relaxation introduced in the space $BV$
by Bouchitté, Fonseca \& Mascarenhas in \cite{BFM1998}. In the present paper, we show that the ideas in \cite{BMZ2024} work similarly in the space
\begin{equation*}
SD_{L,p}(\Omega):= SBV(\Omega;\mathbb R^{d})\times 
\underbrace{L^1(\Omega;\mathbb R^{d\times N})\times\cdots\times 
L^1(\Omega;\mathbb R^{d\times N})}_{(L-1)\text{-times}} \times 
L^p(\Omega;\mathbb R^{d\times N}),
\end{equation*}
which is relevant for some applications (see Section \ref{app}).

Throughout this paper, 
\begin{itemize}
\item $\Omega \subset \mathbb R^{N}$ is a bounded, connected, open set;
\item $\mathbb S^{N-1}$ denotes the unit sphere in $\mathbb R^N$;
\item $Q\coloneqq (-\tfrac12,\tfrac12)^N$ denotes the open unit cube of 
$\mathbb R^{N}$ centred at the origin; for any $\nu\in\mathbb S^{N-1}$, $Q_\nu$ denotes any open unit cube in $\mathbb R^{N}$ with two faces orthogonal to $\nu$;  for any $x\in\mathbb R^{N}$ and $\varepsilon>0$, $Q(x,\varepsilon)\coloneqq x+\varepsilon Q$ denotes the open cube in $\mathbb R^{N}$ centred at $x$ with side length $\varepsilon$ and $Q_\nu(x,\varepsilon)\coloneqq x+\varepsilon Q_\nu$;
\item ${\mathcal O}(\Omega)$ is the family of all open subsets of $\Omega $, whereas ${\mathcal O}_\infty(\Omega)$ is the family of all open subsets of $\Omega $ with Lipschitz boundary; 
\item $L^p(\Omega;\mathbb R^{d\times N})$ is the set of matrix-valued $p$-integrable functions; 
\item $\mathcal M(\Omega;\mathbb R^{d\times N})$ is the set of finite matrix-valued Radon measures on $\Omega$; $\mathcal M ^+(\Omega)$ is the set of non-negative finite Radon measures on $\Omega$;
given $\mu\in\mathcal M(\Omega;\mathbb R^{d\times N})$,  
the measure $|\mu|\in\mathcal M^+(\Omega)$ 
denotes the total variation of $\mu$;
\item $SBV(\Omega;\mathbb R^d)$ is the set of vector-valued \emph{special functions of bounded variation} defined on $\Omega$. 
Given $u\in SBV(\Omega;\mathbb R^d)$, its distributional gradient $Du$ admits the decomposition 
$$Du=D^au+D^su=\nabla u\cL^N+[u]\otimes\nu_u\mathcal H^{N-1}\res S_u,$$ 
where $S_u$ is the jump set of~$u$, $[u]$ denotes the jump of~$u$ on $S_u$, and $\nu_u$ is the unit normal vector to $S_u$.
\end{itemize}

The paper is organised in the following way. In Section \ref{SDdef} we recall the notion of structured deformation, first in the $L^{\infty}$ setting, as introduced by Del Piero $\&$ Owen in \cite{DPO1993}, and then its extension to the $SBV$ context by Choksi $\&$ Fonseca \cite{CF1997}. Section \ref{CFrel} is devoted to recalling the relaxation result obtained in \cite{CF1997} and later generalised to the inhomogeneous case in \cite{MMOZ}. 
Based on \cite{DO2019}, in Section \ref{mlsd} we recall the notion of a multi-level structured deformation, as well as the notion of convergence of a (multi-indexed) sequence to such a deformation, and a version of the approximation theorem in this case. In Section \ref{gm} we revisit the global method for relaxation for multi-level structured deformations proved in \cite{BMZ2024} and provide a generalisation to the case where the field $G_L$ may exhibit a different integrability behaviour from the remaining fields $(G_1, \ldots, G_{L-1})$.
Finally, Section \ref{app} contains some applications of the results of the previous section.

\section{Notion of (first-order) structured deformations}\label{SDdef}

\subsection{The notion of Del Piero \& Owen}\label{DPO}

First-order structured deformations were introduced by Del Piero $\&$ Owen in \cite{DPO1993} to address problems in non-classical deformations of continua (for instance, study of equilibrium configurations of crystals with defects) where an analysis at both the macroscopic and a microscopic level is required. This theory
provides a mathematical framework that captures the effects at the ma\-cros\-co\-pic level of both smooth deformations and of non\--\-smooth deformations, the so-called disarrangements, that occur at one sub-macroscopic level.

In the classical theory of mechanics, the deformation of the body is characterised exclusively by the ma\-cros\-co\-pic deformation field, $g$, and its gradient, $\nabla g$.
However, in the framework of structured deformations, an additional geometrical field, $G$, is introduced in order to capture the contributions at the ma\-cros\-co\-pic scale of smooth sub-macroscopic geometrical changes such as stretching, shearing and rotation. This broad theory can address phenomena such as elasticity, plasticity and the behaviour of crystals with defects.

Precisely, a first-order structured deformation from a region $\Omega \subset \mathbb{R}^N$ is a triple $(K,g, G)$ where
$g: \Omega \to \Rb^d$ is the macroscopic deformation field and accounts for
macroscopic changes in the geometry of the body, this field is assumed to be injective and piecewise smooth,
$G : \Omega \to \Rb^{d\times N}$ is the geometrical field mentioned above, it is assumed to be piecewise continuous and, finally,
$K$ is a surface-like subset of the body that describes pre-existing unopened macroscopic cracks.

In addition, the following accommodation inequality holds
$$0 < c \leq \det G(x) \leq \det\nabla g(x), 
\; \forall x \in \Omega.$$
As the geometric interpretation of $G$, below, shows, $\det G$ is the volume change without disarrangements, whereas $\det \nabla g$ represents the macroscopic volume change, the inequality 
$\det G(x) \leq \det\nabla g(x)$ tells us that volume changes associated with smooth sub-macroscopic geometrical changes cannot exceed volume changes associated with smooth macroscopic geometrical changes. This condition is necessary to ensure that the $u_n$ approximating $g$ (see Theorem \ref{approxDPO} that follows) are injective and hence interpenetration of matter is avoided.

Crucial to the theory is the following approximation theorem obtained in 
\cite{DPO1993}.

\begin{theorem}[Approximation Theorem - Del Piero \& Owen]\label{approxDPO}
For each structured deformation $(K,g, G)$
there exists a sequence, $u_n : \Omega \to \Rb^d$, 
of injective and piecewise smooth mappings, 
such that 
$u_n\to g$ and $\nabla u_n\to G$, in $L^{\infty}$.
\end{theorem}

The sequence $u_n$ whose existence is guaranteed by the approximation theorem is referred to as a determining sequence. 

The above result provides a geometrical interpretation of the field $G$. Indeed, 
since the approximation theorem shows that it is a limit of gradients, 
$G : \Omega \to \Rb^{d\times N}$ is not influenced by any 
discontinuities associated with the piecewise smooth mappings $u_n$ so it is called
the deformation without disarrangements.

On the other hand, the difference $\nabla g-G$  captures the contribution at the macroscopic scale of non-smooth 
changes, such as slips and separations (which are called disarrangements),
that take place at a smaller length scale, so  $\nabla g-G$ is called the deformation due to disarrangements. In fact, as was proved by Del Piero \& Owen \cite{DPO1995}, it turns out that at every point where $g$ is differentiable and $G$ is continuous,
the following equality holds
$$\lim_{r\to 0^+} \lim_{n \to +\infty} 
\frac {\displaystyle \int_{S(u_n)\cap B(x, r)} [u_n](y) \otimes \nu_{u_n}(y) \, d \mathcal{H}^{N-1}(y)}{|B(x, r)|} =  \nabla g(x) - G(x),
$$
where $ [u_n](y) \otimes \nu_{u_n}(y)$ is the tensor product of the jump of $u_n$ at $y$ and the normal to the jump set of $u_n$ at $y$.
Thus, the difference $\nabla g(x) - G(x)$ is the limit of averages of the directed jumps $[u_n](y) \otimes \nu_{u_n}(y)$ in the approximating mappings. This shows that only the non-smooth part of $u_n$ affects the value of $M(x) := \nabla g(x) - G(x)$
which is, therefore, called the deformation due to disarrangements.

Hence we may write $\nabla g = G + M$, thus expressing the macroscopic deformation gradient as the sum of the term $G$, the deformation without disarrangements,
which is associated with limits of gradients of approximating deformations, and the term $M$, the deformation due to disarrangements, which is associated to their jump effects. 

Taking into account the approximation theorem, we also see that
$$M(x) := \nabla g(x) - G(x)= \nabla \lim u_n - \lim \nabla u_n,$$ 
stressing the fact that, in this context, the classical gradient and the limit do not commute.

\subsection{The notion of Choksi \& Fonseca}\label{CF}

Given a body undergoing internal stresses and external loads, 
a central problem in the calculus of variations is to identify its equilibrium configurations via an energy minimisation process. To this end, we need to understand what is meant by the energy of a structured deformation.
This question was first addressed by Choksi \& Fonseca \cite{CF1997} within the broader notion of structured deformation based on $SBV$ functions. Indeed, 
since the singular part of their distributional derivative is supported precisely on the set where the function has jump discontinuities, these fields are generalisations of piecewise smooth functions, and they are better suited to apply calculus of variations arguments.
 
Thus, the variational formulation for (first-order) structured deformations in the $SBV$ setting is based on the following notion 
(see \cite{CF1997}): a (first-order) structured deformation is a pair 
$(g,G)\in SBV(\Omega;\mathbb R^{d})\times L^1(\Omega;\mathbb R^{d\times N})$ 
where $g \in  SBV(\Omega; \Rb^d)$  is the macroscopic deformation and
$G \in  L^1(\Omega; \Rb^{d\times N})$  is the part of the deformation without disarrangements. 

In this case, the jump set $S_g$ can be viewed as the crack site, the role previously played by the surface-like set $K$.
Notice also that $\nabla u$ is no longer the classical gradient of a smooth field so it need not be curl-free.
  
Therefore, starting from a functional which associates to any deformation $u$ of the body an energy featuring a bulk contribution, that measures the deformation (gradient) throughout the whole body, and an interfacial contribution that accounts for the energy needed to fracture the body,
\begin{equation}\label{energydef}
E(u) = \int_{\Omega} W(\nabla u(x))\, dx + 
\int_{S_u}\psi([u](x), \nu_u(x))\, d{\mathcal H}^{N-1},
\end{equation}
for certain bulk and interfacial energy densities $W$ and $\psi$,
the question that now arises is how to assign an energy to a structured deformation $(g,G)$.

To provide an answer to this question, a counterpart of the approximation theorem of Del Piero \& Owen, with respect to a weaker topology, was obtained in \cite{CF1997}.  Its proof is based on the fact that $BV$ functions can be approximated in the $L^1$-norm by sequences of piecewise constant functions and on the following well-known result due to Alberti \cite{AL}.

\begin{theorem}[Alberti]\label{Alb}
Given $f \in L^1(\Omega; \Rb^{d\times N})$,
there exists $u \in SBV (\Omega; \Rb^d)$ such that
$\nabla u = f$, ${\mathcal L}^N$ a.e. in $\Omega$,
$|D^su|(\Omega) \leq C \|f\|_{L^1(\Omega; \Rb^{d\times N})}$ and
$\|u\|_{L^1(\Omega; \Rb^d)}\leq C\|f\|_{L^1(\Omega; \Rb^{d\times N})}.$
\end{theorem}

The Choksi \& Fonseca approximation result reads as follows.

\begin{theorem}[Approximation Theorem - Choksi \& Fonseca]\label{approxthmCF}
Given $(g,G) \in SBV(\Omega; \Rb^d) \times L^1(\Omega; \Rb^{d\times N})$,
there exists a sequence $u_n \in SBV(\Omega; \Rb^d)$ such that 
$u_n \to g$ in $L^1$ and $\nabla u_n \weakst G$ 
in the sense of measures.
\end{theorem} 
\begin{proof}
Given $G \in L^1(\Omega; \Rb^{d \times N})$ and 
$g \in SBV(\Omega; \Rb^d)$, using Theorem \ref{Alb},  let 
$h \in SBV (\Omega; \Rb^d)$ be such that $\nabla h = G - \nabla g,$
and let $h_n$ be a piecewise constant approximation of $h$ in the $L^1$ norm.
Defining  $u_n = g + h - h_n$ it is immediate to see that  
$u_n \to g $ in $L^1$ and $\nabla u_n = G,$ thus satisfying the requirements of the statement.
\end{proof}

Denoting by $E(\cdot)$ the energy associated to any deformation of the body,
given in \eqref{energydef},
it would be natural to consider the limit $\lim E(u_n)$ as the energy
of the structured deformation $(g,G)$, where $u_n$ is a determining
sequence whose existence is ensured by Theorem \ref{approxthmCF}.
However, determining sequences are far from unique and the limit  $\lim E(u_n)$  might depend on the choice of the determining sequence, so we are lead to consider a
relaxation procedure.

Thus, due to the approximation theorem, and given the lack of uniqueness of approximating sequences, the energy associated with the structured deformation $(g,G)$ is defined as the most effective way to build up the deformation using approximations in $SBV$, i.e. 
as the relaxation, with respect to the topology considered in Theorem \ref{approxthmCF}, of $E(\cdot)$:
\begin{equation}\label{CFrelen}
I(g,G) =  \inf_{u_n\in SBV(\Omega; \Rb^d)} \left \{ \liminf_{n\to +\infty} E(u_n) : \,\,u_n\debaixodaseta {L^1}{}  g, \,\,  \nabla u_n \weakst G \right\},
\end{equation}
where $E(\cdot)$, given by \eqref{energydef},
is the energy assigned to each $u \in SBV(\Omega;\Rb^d)$.

An important feature of this relaxation problem is that the gradients of the approximating sequences $u_n$ are constrained to converge, in the sense of measures, to the given function $G$, which is not necessarily equal to $\nabla g$. 
On the other hand, since $u_n\debaixodaseta {L^1}{}  g$, then $Du_n \weakst Dg$, so 
$D^s u_n \weakst  Dg - G$, in the sense of measures. If there are no macroscopic cracks, i.e. if $g \in W^{1,1}$, then
$Dg = \nabla g$ so the difference between the macroscopic and the
microscopic bulk fields, $\nabla g - G$, is achieved by the limit of singular measures.

Moreover, by a compactness theorem in $SBV$ due to Ambrosio \cite{Amb}, we have that $\nabla g = G$,  $\mathcal{L}^N$ a.e., unless $\mathcal{H}^{N-1}(S(u_n)) \to +\infty$, that is, unless there is a diffusion of cracks whose amplitude is tending to zero.

As we will recall in the next section, in their seminal paper \cite{CF1997},
Choksi \& Fonseca obtained an integral representation formula for $I(g,G)$ and identified the corresponding relaxed energy densities under certain hypotheses on the original densities $W$ and $\psi$ (cf. Section \ref{CFrel}).


\section{Choksi \& Fonseca Relaxation}\label{CFrel}

Our aim in this section is to recall the integral representation formula for $I(g,G)$, obtained in \cite{CF1997}
in the homogeneous case, and generalised by Matias, Morandotti, Owen \& Zappale 
in \cite{MMOZ}, to the case of a uniformly continuous $x$ dependence of the densities
$W$ and $\psi$.

Let $p \geq 1$, let $\Omega \subset \mathbb R^N$ be a bounded, open set and consider the initial energy of $u\in SBV(\Omega;\mathbb R^d)$ defined by
\begin{equation}\label{Ex}
E(u):= \int_\Omega W(x,\nabla u(x))\,dx
+\int_{\Omega\cap S_u} \psi(x,[u](x),\nu_u(x))\,d\cH^{N-1}(x),
\end{equation}
where $W\colon\Omega\times\mathbb R^{d\times N}\to[0,+\infty)$ and $\psi\colon\Omega\times\mathbb R^{d}\times\mathbb S^{N-1}\to[0,+\infty)$ are continuous functions satisfying the following set of hypotheses:
\begin{enumerate}
\item[(W1)] ($p$-Lipschitz continuity) there exists $C_W >0$ such that, for all $x\in\Omega$ and $A_1,A_2 \in \mathbb R^{d\times N}$,
\begin{equation*}
|W(x,A_1) - W(x,A_2)| \leq C_W |A_1 - A_2| \big(1+|A_1|^{p-1}+|A_2|^{p-1}\big);
\end{equation*}
\item[(W2)] (modulus of continuity) there exists a continuous function $\omega_W\colon[0,+\infty)\to[0,+\infty)$, with $\omega_W(s)\to 0$ as $s\to0^+$, such that, for every $x_0,x_1\in\Omega$ and $A \in \mathbb R^{d \times N}$,
\begin{equation*}
|W(x_1,A)-W(x_0,A)|\leq\omega_W(|x_1-x_0|)(1 + |A|^p);
\end{equation*}
\item[(W3)] if $p=1$, there exist $C, L >0$ and $0 < \alpha < 1$ such that, for every $x \in \Omega$ and every $A \in  \mathbb R^{d\times N}$, with $|A| = 1$,
\begin{equation*}
\left|W^{\infty}(x,A)-\frac{W(x,tA)}{t}\right|\leq \frac{C}{t^\alpha}, \; 
\forall t > L
\end{equation*}
where $W^{\infty}$ denotes the recession function of $W$ with respect to the second variable given by
\begin{equation}\label{recfun}
W^{\infty}(x,A) := \limsup_{t \to + \infty}\frac{W(x,tA)}{t}, \; \forall x \in \Omega, \forall A \in \mathbb R^{d \times N}; 
\end{equation}
\end{enumerate}
\begin{enumerate}
\item[($\psi 1$)] (symmetry) $\forall x \in \Omega$, 
$\lambda \in \mathbb R^{d}$ and 
$\nu \in \mathbb S^{N-1}$, 
\begin{equation*}
\psi (x, \lambda, \nu)= \psi (x,-\lambda, -\nu);
\end{equation*}
\item[($\psi 2$)] (growth) $\exists \, c_\psi,C_\psi > 0$ such that,  $\forall x\in\Omega$, 
$\lambda \in \mathbb R^{d}$, $\nu \in \mathbb S^{N-1}$,
\begin{equation*}
c_\psi|\lambda| \leq \psi(x,\lambda, \nu) \leq C_\psi|\lambda |;
\end{equation*}
\item[($\psi 3$)] (positive $1$-homogeneity) $\forall x\in\Omega$, $\lambda \in \mathbb R^{d}$, $\nu \in \mathbb S^{N-1}$, $t >0$,
$$\psi(x,t\lambda, \nu) = t\psi(x, \lambda, \nu);$$
\item[($\psi 4$)] (sub-additivity) $\forall x\in\Omega$, $\lambda_1,\lambda_2 \in \mathbb R^{d}$, $\nu \in \mathbb S^{N-1}$,
\begin{equation*}
\psi(x, \lambda_1 + \lambda_2, \nu) \leq 
\psi(x,\lambda_1, \nu) +\psi(x,\lambda_2, \nu);
\end{equation*}
\item[($\psi 5$)] (modulus of continuity) there exists a continuous function $\omega_\psi\colon[0,+\infty)\to[0,+\infty)$, with $\omega_\psi(s)\to 0$ as $s\to0^+$, such that, for every $x_0,x_1\in\Omega$, $\lambda \in \mathbb R^{d}$ and 
$\nu \in \mathbb S^{N-1}$,
\begin{equation*}
|\psi(x_1,\lambda,\nu)-\psi(x_0,\lambda,\nu)|\leq\omega_\psi(|x_1-x_0|)|\lambda|.
\end{equation*}
\end{enumerate}

Given $(g, G)\in SBV(\Omega;\mathbb R^{d})\times L^1(\Omega;\mathbb R^{d\times N})$, 
consider the relaxed functional defined by
\begin{align}\label{Ip}
I_p(g,G):= \inf\Big\{\liminf_{n\to\infty} E(u_n) &: u_n \in SBV(\Omega;\mathbb R^{d}), u_n\debaixodaseta {L^1}{}  g, \,\,  \nabla u_n \weakst G, \nonumber\\
&\hspace{1cm}(1 - \delta_1(p))\sup_n\|\nabla u_n\|_{L^p(\Omega;\mathbb R^{d})} < + \infty \Big\},
\end{align}
where $\delta_1(p) = 1$ if $p=1$ and $\delta_1(p) = 0$ otherwise.
The requirement that the density $W$ be coercive in the second variable can be avoided due to the uniform bound placed on the gradients of the approximating sequences.

We point out that in the case where $W$ and $\psi$ do not depend explicitly on $x$,
the hypotheses above (except for $(W2)$ and $(\psi 5)$), and the functional 
$I_p(g,G)$, are those considered by Choksi \& Fonseca in \cite{CF1997} and for which they proved their integral representation result.

Under the previous hypotheses, it was shown in \cite[Theorem 5.1]{MMOZ}
that $I_p(g,G)$ admits an integral representation,
that is, there exist functions 
$H_p\colon\Omega\times\R{d\times N}\times\R{d\times N}\to[0,+\infty)$ and $h_p\colon\Omega\times \R{d}\times\S{N-1}\to[0,+\infty)$ such that
\begin{align*}
I_p(g,G) =\int_\Omega H_p(x,\nabla g(x),G(x))\, dx 
+ \int_{\Omega\cap S_g} h_p(x, [g](x),\nu_g(x))\, d\cH^{N-1}(x),
\end{align*}
and the limit densities $H_p$ and $h_p$ were identified as follows.

For $A,B\in\mathbb R^{d\times N}$ let
\begin{align*}
\cC_p^{\bulk}(A,B):= \bigg\{u\in SBV(Q;\mathbb R^{d}): u|_{\partial Q}(x)=Ax, \int_Q \nabla u\, dx=B, |\nabla u|\in L^p(Q) \bigg\}, 
\end{align*}
for $\lambda\in\mathbb R^{d}$ and $\nu\in\mathbb S^{N-1}$ let
 $u_{\lambda,\nu}$ be the function defined by
\begin{equation*}
u_{\lambda,\nu}(x):=
\begin{cases}
\lambda & \text{if $x\cdot\nu\geq0$,} \\
0 & \text{if $x\cdot\nu<0$,}
\end{cases}
\end{equation*}
and, for $p > 1$,
\begin{align*}
\cC_p^\surface(\lambda,\nu):= \Big\{u\in SBV(Q_\nu;\mathbb R^{d}): 
u|_{\partial Q_\nu}(x)=u_{\lambda,\nu}(x),
\nabla u(x)=0\;\text{for $\cL^N$-a.e.~$x\in Q_\nu$}\Big\},
\end{align*}
whereas
\begin{align*}
\cC_1^\surface(\lambda,\nu):= \Big\{u\in SBV(Q_\nu;\mathbb R^{d}): 
u|_{\partial Q_\nu}(x)=u_{\lambda,\nu}(x),
\int_{Q_\nu}\nabla u(x) \, dx=0\Big\}.
\end{align*}

Then, the relaxed densities $H_p$ and $h_p$ are given by 
\begin{align}\label{Hp}
H_p(x_0,A,B):= \inf\bigg\{  
\int_Q W(x_0,\nabla u(x))\, dx
+\int_{Q\cap S_u} \psi(x_0,[u](x),\nu_u(x))\,d\cH^{N-1}(x) 
: u\in\cC_p^\bulk(A,B)\bigg\},
\end{align}
for all $x_0\in\Omega$ and $A,B\in\mathbb R^{d\times N}$,
and, for all $x_0\in\Omega$, $\lambda\in\mathbb R^{d}$ and $\nu\in\mathbb S^{N-1}$,
\begin{equation}\label{hp}
h_p(x_0,\lambda,\nu):= \inf\bigg\{ 
\delta_1(p) \!\! \int_{Q_\nu} \!\!\!\! W^\infty(x_0,\nabla u(x))\, dx+ \! 
\int_{Q_\nu\cap S_u} \!\!\!\!\!\!\!\!\! \psi(x_0,[u](x),\nu_u(x))\, d\cH^{N-1}(x): 
u\in\cC_p^\surface(\lambda,\nu)\bigg\}.
\end{equation}

The relaxed bulk density $H_p$ exhibits an interaction between the initial densities $W$ and $\psi$, this tells us that the jumps in the approximating sequences diffuse throughout portions of the body and in the limit contribute to both bulk and surface terms.
If $p > 1$ and gradients of admissible sequences are bounded in $L^p$, the relaxed interfacial density $h_p$ is independent of $W$, this means that it is cheaper to approximate jumps with jumps rather than with sharp gradients. On the other hand, if $p = 1$, there is a contribution of $W$, through its recession function $W^{\infty}$, in $h_1$.

\section{Multi-level Structured Deformations}\label{mlsd}

A further step in the development of the theory of structured deformations was undertaken by Deseri \& Owen \cite{DO2019}, where the notion was extended to so-called hierarchical systems of structured deformations, in order to include the effects of disarrangements that occur at more than one sub-macroscopic level.
Indeed, as it happens, many natural and man-made materials, for example, muscles, cartilage, bones, plants and some biomedical materials, do, in fact, exhibit  different levels of disarrangements.

In this broader setting, 
a (first-order) structured deformation $(g,G)$ corresponds to a two-level hierarchical system, the macroscopic level, expressed through the field $g$, plus one microscopic level, related to the field $G$.
For $L \in \mathbb N$, $L>1$, an 
($L+1$)-level hierarchical system of structured deformations consists of an $(L+1)$-tuple $(g, G_1, \ldots, G_{L})$, where each $G_i, i = 1, \ldots L$,  provides the effects at the macro-level of the corresponding sub-macroscopic level $i$.

For $L\in\mathbb N$, $p\geq 1$ and $\Omega\subset\mathbb R^{N}$ a bounded, connected, open set, we define
$$HSD_L^p(\Omega):= SBV(\Omega;\mathbb R^{d})\times \underbrace{L^p(\Omega;\mathbb R^{d\times N})\times\cdots\times L^p(\Omega;\mathbb R^{d\times N})}_{L\text{-times}},$$
the set of $(L+1)$-level (first-order) structured deformations on $\Omega$. 

In the remainder of this article the presented results are stated for $L = 2$, for
simplicity and brevity of exposition, for the general case of any $L \in \mathbb N$
we refer to \cite{BMMOZ}, \cite{BMZ2024} and \cite{BMMOZ2}.

We begin by defining the notion of convergence for a three-level structured deformation.

\begin{definition}\label{mlconv}
We say that the (double-indexed) sequence 
$u_{n_1,n_2}\in SBV(\Omega;\mathbb R^{d})$ converges in the sense of $HSD_2^p(\Omega)$ to $(g,G_1,G_2)$ if
\begin{itemize}
\item[(i)] $\displaystyle\lim_{n_1\to+\infty}\lim_{n_2\to+\infty} u_{n_1,n_2} = g$,
where the iterated limit holds in the sense of $L^1(\Omega;\mathbb R^{d})$;
\item[(ii)] 
$\displaystyle \lim_{n_{2}\to+\infty}u_{n_1,n_2} =: g_{n_1} \in SBV(\Omega;\mathbb R^{d})$, where this limit is in the sense of $L^1(\Omega;\mathbb R^{d})$ convergence, and 
$\displaystyle \lim_{n_1\to+\infty}\nabla g_{n_1}=G_{1},$ 
where the limit holds in the sense of weak convergence in  $L^p(\Omega;\mathbb R^{d\times N})$, if $p > 1$, and 
in the sense of measures, if $p=1$;
\item[(iii)] $\displaystyle \lim_{n_1\to+\infty}\lim_{n_2\to+\infty} \nabla u_{n_1,n_2} = G_2$,  
where the iterated limit holds in the sense of weak convergence in  $L^p(\Omega;\mathbb R^{d\times N})$, if $p > 1$, and in the sense of measures, if $p=1$.
\end{itemize}
We use the notation 
$$u_{n_1,n_2}\pweak(g,G_1,G_2)$$ 
to indicate this convergence.
\end{definition}

The following version of the approximation theorem, stated here for three-level structured deformations, was obtained in \cite{BMMOZ}.

\begin{theorem} \label{gapth}
For any three-level structured deformation $(g,G_1,G_2) \in HSD_2^p(\Omega)$ 
there exists a (double-indexed) sequence 
$(n_1,n_2)\mapsto u_{n_1,n_2}\in SBV(\Omega;\mathbb R^{d})$
converging to $(g,G_1,G_2)$ in the sense of Definition \ref{mlconv}.
\end{theorem}
\begin{proof}
Let $u_1, u_2 \in SBV(\Omega;\mathbb R^{d})$ be 
functions, provided by Theorem~\ref{Alb}, 
satisfying $\nabla u_1= \nabla g - G_1$ and $\nabla u_2 = G_{1}-G_2$ 
and, for $\ell = 1,2$, let $n_\ell\mapsto \overline u_{n_\ell}$ be piecewise constant sequences approximating $u_\ell$ in $L^1(\Omega;\mathbb R^{d})$.
Then the (double-indexed) sequence given by
$$\displaystyle u_{n_1,n_2} := g+ (\overline u_{n_1}-u_1) + (\overline u_{n_2}-u_2)$$
verifies the desired properties.
\end{proof}

We point out that in the previous proof, the order in which the double limits are taken is crucial to attain the required result. 

In analogy with \eqref{Ip}, Theorem \ref{gapth} may be used to assign an energy to a three-level structured deformation $(g,G_1,G_2)$, see Section \ref{iterel} for more details.

\section{The Global Method for Relaxation}\label{gm}

This is the main section of our paper. In it we present the integral representation 
result for a general class of abstract functionals, defined in the space of
multi-level structured deformations, obtained in \cite{BMZ2024}. In fact, in
Corollary \ref{glm2} below, we present an improvement of \cite[Theorem 3.2]{BMZ2024}
in the sense that we now allow for different integrability behaviours for the fields corresponding to different sub-macroscopic levels.

The results mentioned herein are based on the ideas of the global method for relaxation, introduced by Bouchitté, Fonseca \& Mascarenhas in \cite{BFM1998} to obtain an integral representation for an abstract functional $\mathcal F(u;O)$,
defined for $u \in BV(\Omega;\mathbb R^d)$ and $O$ an open subset of $\Omega$, such that the set function $\mathcal F(u;\cdot)$ is the restriction to the open subsets of $\Omega$, of a Radon measure and $\mathcal F(u;O)$ satisfies certain
lower semicontinuity, locality and growth conditions. See \cite{BFM1998} for more details or $(H1)-(H4)$ below for the precise hypotheses in our setting.

Given the Dirichlet-type functional
$$m(u;O) := \inf\left\{\mathcal F(v;O) : v \in BV(\Omega;\mathbb R^d), 
u = v \mbox { near } \partial O\right\},$$
defined for $u \in BV(\Omega;\mathbb R^d)$ and $O \in \mathcal O(\Omega)$, the method is based on the fact that $\mathcal F(u;O)$ and $m(u;O)$ behave in a similar fashion when $O$ is a cube of small side-length, so the relaxed energy densities that represent $\mathcal F(u;O)$ are characterised in terms of $m$ through a blow-up argument.

Since its inception, this global method for relaxation has known numerous applications and generalisations, in particular it was used in the context of (second-order) structured deformations in the space $BH$ by
Fonseca, Hagerty \& Paroni in \cite{FHP}.

The results of Theorem \ref{BMZgm} and Corollary \ref{glm2} are quite general 
and cover, with the same proof, a hierarchical system of structured deformations with an arbitrary number of sub-macroscopic levels, with no need for iterative procedures (cf. Section \ref{iterel} for more comments on this topic).

We begin by introducing the space of test functions. For $(g, G_1, G_2)\in HSD_2^p(\Omega)$ and $O \in \mathcal O_{\infty}(\Omega)$, let
\begin{align*}
\mathcal C_{HSD^p_2}(g, G_1,G_2; O):=\Big\{(u, U_1,U_2)\in HSD^p_2 (\Omega) &: u=g 
 \hbox{ in a neighbourhood of } \partial O, \\
&\int_O (G_i-U_i) \,dx =0, i=1,2  \Big\},
\end{align*}
and consider the Dirichlet-type functional, $m : HSD^p_2(\Omega)\times \mathcal O_\infty(\Omega) \to [0,+\infty]$, defined by
\begin{align}\label{Dirf}
m(g, G_1,G_2;O):=\inf\Big\{\mathcal F(u, U_1,U_2; O):
(u, U_1, U_2)\in 
\mathcal C_{HSD^p_2}(g, G_1, G_2; O)\Big\},
\end{align}
where, for $p \geq1$, $\mathcal F: HSD_2^p(\Omega) \times\mathcal 
O(\Omega)\to [0, +\infty]$ is a functional satisfying 
\begin{enumerate}
\item[$(H1)$] (measure)
for every $(g, G_1, G_2) \in HSD^p_2(\Omega)$, 
$\mathcal F(g, G_1,G_2;\cdot)$ is the restriction to 
$\mathcal O(\Omega)$ of a Radon measure; 
\item[$(H2)$] (lower semicontinuity)
for every $O \in \mathcal O(\Omega)$, 
\begin{itemize}
\item if  $p > 1$, 
$\mathcal F(\cdot, \cdot, \cdot; O)$ is $HSD^p_2$-lower semicontinuous,  
that is, 
if $(g, G_1,G_2)\in HSD_2^p(\Omega)$ and  $\left((g^n, G_1^n,G^n_2)\right) \subset HSD_2^p(\Omega)$ are such that $g^n \to g$ in $L^1(\Omega;\mathbb R^{d} )$, $G^n_i\rightharpoonup G_i$ in $L^p(\Omega;\mathbb R^{d \times N})$, for $i=1,2$, then
$$
\mathcal F(g, G_1,G_2;O)\leq \liminf_{n\to +\infty}\mathcal F(g^n,G_1^n, G^n_2;O);$$
\item the same holds in the case $p=1$, replacing the weak convergences
$G^n_i\rightharpoonup G^i$ in $L^p(\Omega;\mathbb R^{d \times N})$, for $i=1,2$, with weak star convergences in the sense of measures 
$\cM(\Omega;\mathbb R^{d\times N})$;
\end{itemize}
\end{enumerate}
\begin{enumerate}
\item[$(H3)$] (locality)
for all $O \in \mathcal O(\Omega)$, $\mathcal F(\cdot, \cdot, \cdot;O)$ is local, that is, if $g= u$, $G_1= U_1$, $G_2=U_2$ a.e. in $O$, then 
$$\mathcal F(g, G_1,G_2;O)= \mathcal F(u, U_1, U_2;O);$$
\item[$(H4)$] (growth) there exists $C>0$ such that
\begin{align*}
\frac{1}{C}\left(\sum_{i=1}^2 \||G_i|^p\|_{L^1(O;\mathbb R^{d \times N})}+ |D g|(O)\right)&\leq \mathcal F(g, G_1, G_2;O)\\
&\leq
C\left(\mathcal L^N(O) +\sum_{i=1}^2 \||G_i|^p\|_{L^1(O;\mathbb R^{d \times N})}+  |D g|(O)\right),
\end{align*}
for every $(g, G_1, G_2) \in HSD^p_2(\Omega)$ and every $O \in \mathcal O(\Omega)$.
\end{enumerate}

Based on the set of hypotheses listed above, the following result was proved in
\cite{BMZ2024}. We state it here for $L=2$, the general case may be found in \cite[Theorem 3.2]{BMZ2024}.

\begin{theorem}\label{BMZgm}
Let $p \geq 1$ and let $\mathcal F: HSD_2^p(\Omega) \times\mathcal 
O(\Omega)\to [0, +\infty]$ be a functional satisfying (H$1$)-(H$4$). 
Then
\begin{align*}
\mathcal F(u, U_1,U_2;O) &= \int_O \!f(x,u(x), \nabla u(x), U_1(x),U_2(x)) \, dx 
+\int_{O \cap S_u}\!\!\!\!\!\Phi(x, u^+(x), u^-(x),\nu_u(x)) 
\, d\mathcal H^{N-1}(x),
\end{align*}
where, for every $x_0\in \Omega$, $ a, \theta,\lambda \in \mathbb R^d$, 
$\xi, B_1,B_2 \in \mathbb R^{d \times N}$, $\nu \in \mathbb S^{N-1}$,
the relaxed energy densities are given by
\begin{align*}
f(x_0, a, \xi, B_1,B_2) &:=	
\limsup_{\varepsilon \to 0^+}\frac{m(a+ \xi(\cdot-x_0), B_1,B_2; Q(x_0,\varepsilon))}{\varepsilon^N},\\
\Phi(x_0, \lambda, \theta, \nu) &:= \limsup_{\varepsilon \to 0^+}\frac{m(v_{\lambda, \theta,\nu}(\cdot-x_0), 0, 0; Q_{\nu}(x_0,\varepsilon))}{\varepsilon ^{N-1}}.
\end{align*}
In the above expression, $m$ is the functional given in \eqref{Dirf}, $0$ is the zero matrix in $\mathbb R^{d \times N}$
and $v_{\lambda,\theta, \nu}$ is the function defined by
$v_{\lambda,\theta, \nu}(x) := \begin{cases} \lambda, &\hbox{if } x\cdot \nu \geq 0\\
\theta, &\hbox{ if } x\cdot \nu < 0.\end{cases}$
\end{theorem}
\begin{proof}
We present here a sketch of the proof, for more details we refer to \cite{BMZ2024}.

We begin by proving that, if $p>1$, 
\begin{align}\label{ineq}
\limsup_{\delta \to 0^+} m(u, U_1,U_2;Q_\nu(x_0,(1-\delta) r))
\leq m(u, U_1,U_2;Q_\nu(x_0,r)),
\end{align}
where $Q_\nu(x_0, r)$ is any cube centred at $x_0$ with side-length $r$, two faces orthogonal to $\nu$ and contained in $\Omega$. In the case $p=1$, a similar inequality is true for more general sets, other than cubes, with a simpler proof.
 
As in \cite{BFM1998}, the previous inequality leads to the conclusion that, for every $\nu \in \mathbb S^{N-1}$ and for every $(u,U_1,U_2) \in HSD_2^p(\Omega)$,
we have 
\begin{equation}\label{fandm}
\lim_{\varepsilon \to 0^+}
\frac{{\mathcal F}(u, U_1,U_2; Q_\nu(x_0, \varepsilon))}
{\mu( Q_\nu(x_0,\varepsilon))}
= \lim_{\varepsilon \to 0^+} 
\frac{m (u, U_1,U_2; Q_\nu(x_0, \varepsilon))}
{\mu (Q_\nu(x_0,\varepsilon))}
\end{equation}
for $\mu$-a.e. $x_0 \in \Omega$, where $\mu:=\mathcal L^N\lfloor \Omega + |D^s u|,$
thus showing that the measures $\mathcal F(u, U_1,U_2; \cdot)$ and
$m(u, U_1,U_2;\cdot)$ behave in a similar way when applied to cubes with small side-length.

In particular, from equality \eqref{fandm}, we may show that the density of the measure $\mathcal F(u, U_1,U_2; \cdot)$, with respect to each of the measures $\mathcal L^N$ and $|D^su|$, may be computed from that of $m(u, U_1,U_2;\cdot)$ via a blow-up argument and rescaling techniques.
\end{proof}

\begin{remark}\label{inv}
We point out that if $\mathcal F$ is translation invariant in the first variable, i.e., if
$$\mathcal F(u+a, U_1,U_2;O)= \mathcal F(u, U_1,U_2;O),$$ 
for every $((u,U_1, U_2),O) \in HSD^p_2(\Omega)\times \mathcal O(\Omega)$ and for every $a\in \mathbb R^d$,
then the function $f$ does not depend on $a$, and $\Phi$ does not depend on $\lambda$ and $\theta$, but only on the difference $\lambda- \theta$:
\begin{align*}
f(x_0, a, \xi, B_1,B_2) &= f(x_0, 0, \xi, B_1,B_2), \\
\Phi(x_0, \lambda, \theta, \nu) &= \Phi(x_0, \lambda - \theta,0, \nu),
\end{align*}
for all $x_0\in \Omega$, $ a, \theta,\lambda \in \mathbb R^d$, $\xi, B_1,\dots, B_L \in \mathbb R^{d \times N}$ and $\nu \in \mathbb S^{N-1}$. 
In this case we write, with an abuse of notation,
\begin{align*}
f(x_0,\xi, B_1,B_2) &= f(x_0, 0, \xi, B_1,B_2),\\
\Phi(x_0, \lambda - \theta, \nu) &= \Phi(x_0, \lambda - \theta,0, \nu).
\end{align*}
This remark will prove useful for the applications in Section \ref{app}, when confronting the expressions for the relaxed energy densities obtained in Theorem \ref{BMZgm} with those available in the literature.
\end{remark}

We now proceed to show that a similar conclusion to that of Theorem\ref{BMZgm} holds in the space
\begin{equation*}
SD_{L,p}(\Omega):= SBV(\Omega;\mathbb R^{d})\times 
\underbrace{L^1(\Omega;\mathbb R^{d\times N})\times\cdots\times 
L^1(\Omega;\mathbb R^{d\times N})}_{(L-1)\text{-times}} \times 
L^p(\Omega;\mathbb R^{d\times N}).
\end{equation*}
This allows us to consider fields $(g,G_1, \ldots,G_L)$ where $G_L$ may exhibit a different integrability behaviour from the remaining fields $(G_1, \ldots, G_{L-1})$.
As before, for simplicity, we focus here on the case $L=2$ and we refer to  \cite{BMMOZ2} where the result is stated in full generality.

The notion of convergence considered in the space $SD_{2,p}(\Omega)$ is analogous to the one in Definition \ref{mlconv}, except that all the mentioned convergences of
sequences of gradients are now taken in the sense of measures. However, if $p>1$ and 
$$\sup_{n_1}\sup_{n_2}\int_\Omega |\nabla u_{n_1,n_2}|^p \, dx < +\infty,$$
then these limits hold indeed in the sense of $L^p$-weak convergence (cf. \cite[Definition 2.4 and Remark 2.5]{BMMOZ2}). We write 
$u_{n_1,n_2}\weakstH (g,G_1,G_2)$ to indicate this convergence.

\begin{remark}\label{napprox}
As can be easily seen from its proof, Theorem \ref{gapth} still holds in $SD_{2,p}(\Omega)$ with respect to the above notion of convergence.
\end{remark}

\medskip

We begin by stating the set of hypotheses that we consider in this setting.

Let $\mathcal F\colon SD_{2,p}(\Omega) \times\mathcal 
O(\Omega)\to [0, +\infty]$ be a functional satisfying 
\begin{enumerate}
\item[$(H1)$*] (measure) for every $(g, G_1,G_2) \in SD_{2,p}(\Omega)$, 
$\mathcal F(g, G_1,G_2;\cdot)$ is the restriction to 
$\mathcal O(\Omega)$ of a Radon measure; 
\item[$(H2)$*] (lower semi-continuity) for every $O \in \mathcal O(\Omega)$,  
$\mathcal F(\cdot, \cdot, \cdot; O)$ is lower semi-continuous, i.e.,
if $(g, G_1,G_2) \in SD_{2,p}(\Omega)$ and $((g^n, G_1^n, G^n_2)) \subset SD_{2,p}(\Omega)$ are such that $g^n \to g$ in $L^1(\Omega;\mathbb R^{d})$ and  $G_i^n \wsto G_i$ in $\cM(\Omega;\R{d\times N})$, for $i=1,2$, then
$$\mathcal F(g, G_1,G_2;O)\leq 
\liminf_{n}\mathcal F(g^n,G_1^n,G^n_2;O);$$
\item[$(H3)$*] (locality) for all $O \in \mathcal O(\Omega)$, 
$\mathcal F(\cdot, \cdot, \cdot;O)$ is local, that is, if $g= u$, $G_1= U_1$, $G_2=U_2$ a.e. in $O$, then 
$\mathcal F(g, G_1,G_2;O) = \mathcal F(u, U_1,U_2;O)$;
\item[$(H4)$*] (growth) there exists a constant $C>0$ such that
\begin{align*}
&\frac{1}{C}\bigg(\|G_1\|_{L^1(O;\mathbb R^{d \times N})} + 
\|G_2\|^p_{L^p(O;\mathbb R^{d \times N})} + |D g|(O)\bigg)
\leq \cF(g, G_1,G_2;O)\\
&\hspace{2cm}
\leq
C\left(\cL^N(O) + \|G_1\|_{L^1(O;\mathbb R^{d \times N})} 
+ \|G_2\|^p_{L^p(O;\mathbb R^{d \times N})} +  |D g|(O)\right),
\end{align*}
for every $(g, G_1, G_2) \in SD_{2,p}(\Omega)$ and every $O \in \cO(\Omega)$.
\end{enumerate}

Given $(g, G_1, G_2)\in SD_{2,p}(\Omega)$ and 
$O \in \mathcal O(\Omega)$, we introduce the space of test functions 
\begin{align*}
\cC_{SD_{2,p}}(g, G_1, G_2; O)\coloneqq 
\left\{(u, U_1, U_2)\in SD_{2,p}(\Omega): u=g 
\hbox{ in a neighbourhood of } \partial O, \right. \nonumber\\
\left.\int_O (G_i-U_i) \, dx=0, i=1,2  \right\}, 	
\end{align*}
and we let $m_{2,p} \colon SD_{2,p}(\Omega)\times \cO(\Omega)\to[0,+\infty)$ be the functional defined by
\begin{align}\label{m2p}
m_{2,p}(g, G_1,G_2;O):=\inf\Big\{\mathcal F(u, U_1, U_2; O): 
(u, U_1, U_2)\in \mathcal C_{SD_{2,p}}(g, G_1,G_2; O)\Big\}.
\end{align}

Then, the following result holds.

\begin{corollary}\label{glm2}
Let $p \geq 1$ and let 
$\cF\colon SD_{2,p}(\Omega) \times\mathcal O(\Omega)\to [0, +\infty]$ 
be a functional satisfying $(H1)$* - $(H4)$*. Then
\begin{align*}
\mathcal F(u, U_1,U_2;O)&= 
\int_O \!f_{2,p}(x,u(x), \nabla u(x), U_1(x), U_2(x))\, dx \\
&+\int_{O \cap S_u}\!\!\!\!\!\Phi_{2,p}(x, u^+(x), u^-(x),\nu_u(x)) 
\, d\mathcal H^{N-1}(x),
\end{align*}
where
\begin{align*}
f_{2,p}(x_0,a, \xi, B_1, B_2):=	\limsup_{\varepsilon \to 0^+}
\frac{m_{2,p}(a + \xi(\cdot -x_0), B_1,B_2; Q(x_0,\varepsilon))}{\varepsilon^N},
\end{align*}
\begin{align*}
\Phi_{2,p}(x_0, \lambda, \theta, \nu):= \limsup_{\varepsilon \to 0^+}
\frac{m_{2,p}(v_{\lambda, \theta,\nu}(\cdot-x_0), 0,0; Q_{\nu}(x_0,\varepsilon))}{\varepsilon ^{N-1}},
\end{align*}
for a.e. $x_0\in \Omega$, for every $a, \lambda, \theta \in \mathbb R^d$, 
$\xi, B_1, B_2 \in \mathbb R^{d \times N}$, $\nu \in \mathbb S^{N-1}$,
where $0$ is the zero matrix in $\mathbb R^{d \times N}$, 
and $v_{\lambda,\theta, \nu}$ is defined by 
$v_{\lambda,\theta,\nu}(x)\coloneqq 
\begin{cases}
\lambda & \text{if } x \cdot \nu \geq 0, \\
\theta & \text{if } x \cdot \nu < 0.
\end{cases}$
\end{corollary}
\begin{proof} The proof is identical to the one in \cite[Theorem 3.2]{BMZ2024}, upon noticing that the fields $G_i, i = 1,2$, behave independently of each other.

Indeed, the new growth condition $(H4)$* impacts \cite[Remark 3.1]{BMZ2024} which now reads
\begin{align*}
\mathcal F(u, U_1,U_2;O_2) \leq & \, \mathcal F(u, U_1,U_2;O_1) \\
&+C\left(\mathcal L^N(O_2 \setminus O_1) +
\|U_1\|_{L^1(O_2\setminus O_1;\mathbb R^{d \times N})}
+  \|U_2\|^p_{L^p(O_2\setminus O_1;\mathbb R^{d \times N})}
 +  |D u|(O_2 \setminus O_1)\right),
\end{align*}
for any $(u, U_1,U_2) \in SD_{2,p}(\Omega)$ and any open sets 
$O_1 \subset \subset O_2 \subseteq \Omega$. 

This, in turn, is used in the proof of 
\cite[Lemma 3.4]{BMZ2024}, in order to show
inequality \eqref{ineq}. As the estimate in this lemma is obtained
separately for each field corresponding to each of the different sub-macroscopic levels, and the proof contemplates both the cases $p>1$ and $p=1$, the conclusion still holds in the present context. From here, \cite[Theorem 3.6]{BMZ2024} follows as before.

Finally, to obtain the expressions for the relaxed energy densities in terms of the functional $m_{2,p}$, it suffices to follow the reasoning given in the proof of
\cite[Theorem 3.2]{BMZ2024} where, as above, the estimates therein are obtained separately for each of the fields $U_1$, $U_2$ (cf. \cite[equations (3.19) and (3.32)]{BMZ2024}).
\end{proof}

The same comment presented in Remark \ref{inv} applies also in this situation.
In addition, we note that,
for $L=2$, Corollary \ref{glm2} provides a variant to Theorem \ref{BMZgm} only when $p>1$, as in the case $p=1$ the function spaces $HSD^p_1(\Omega)$ and $SD_{2,1}(\Omega)$ coincide.

\section{Applications}\label{app}
\subsection{The Choksi \& Fonseca Case Revisited}\label{CFapp}

As a first application of Theorem \ref{BMZgm} for $L=1$, we re-examine the 
Choksi $\&$ Fonseca relaxation, which is an example set in the space of two-level structured deformations. We will refer to its generalisation to the inhomogeneous case, given in \cite{MMOZ}, and considered in Section \ref{CFrel}.

Through our global method strategy, and given the energy in \eqref{Ex}, we obtain an integral representation result for
\begin{align}\label{newIp}
I_p(g,G):= \inf\Big\{\liminf_{n\to\infty} E(u_n) &: u_n \in SBV(\Omega;\mathbb R^{d}), u_n\debaixodaseta {L^1}{}  g, \,\,  \nabla u_n \weakst G \Big\},
\end{align}
under a similar set of hypotheses on $W$ and $\psi$ but requiring only measurability, rather than uniform continuity, of $W$ in the $x$ variable. From the modelling point of view, this setting is more realistic since it allows for the consideration of materials that may exhibit very different behaviours from point to point, such as multi-grain type materials or other types of mixtures.

Precisely, we assume that
$W\colon\Omega \times \mathbb R^{d \times N}\to[0,+\infty)$ 
is a Carath\'eodory function such that: 
\begin{enumerate}
\item[$(W1)$*] ($p$-Lipschitz continuity) there exists $C_W >0$ such that, for a.e. $x\in\Omega$ and for all $A_1,A_2 \in \mathbb R^{d\times N}$,
\begin{equation*}
|W(x,A_1) - W(x,A_2)| \leq C_W |A_1 - A_2| \big(1+|A_1|^{p-1}+|A_2|^{p-1}\big);
\end{equation*}
\item[$(W4)$] (bound) there exists $A_0 \in \mathbb R^{d \times N}$ such that 
$W(\cdot, A_0)\in L^\infty(\Omega)$;
\item[$(W5)$] (coercivity) there exists $c_W>0$ such that, for a.e. $x \in \Omega$ and every $A \in  \mathbb R^{d\times N}$,
\begin{equation*}
c_W |A|^{p}-\frac{1}{c_W}\leq W(x,A);
\end{equation*}
\end{enumerate}
and that 
$\psi\colon \Omega \times \mathbb R^d \times \mathbb S^{N-1} \to [0,+\infty)$ 
is a Carath\'eodory function satisfying hypotheses $(\psi 1)-(\psi 5)$ from  
Section \ref{CFrel}. We remark that $(W4)$ and $(W5)$ yield a growth condition from above on $W$ of the form $W(x,A) \leq C(1+|A|^p)$, for some $C>0$ and for a.e.
$x \in \Omega$ and every $A \in  \mathbb R^{d\times N}$.

Under these conditions, the following integral representation for $I_p(g,G)$ was shown in \cite[Theorem 4.1]{BMZ2024}.

\begin{theorem}\label{measCF}
Let $p\geq 1$ and let $\Omega \subset \mathbb R^N$ be a bounded, open set. 
Consider $E$ given by \eqref{Ex} where  
$W\colon\Omega\times\R{d\times N}\to[0,+\infty)$ and $\psi\colon\Omega\times\R{d}\times\S{N-1}\to[0,+\infty)$ satisfy (W$\,1$)*, (W$\,4$), (W$\,5$), ($\psi 1), (\psi 2)$ and $\psi$ is continuous. Let $(g,G) \in SD(\Omega)$, with $G \in L^p(\Omega;\mathbb R^{d \times N})$, 
and assume that $I_p(g,G)$ is defined by \eqref{Ip}. 
Then,  there exist 
$f: \Omega \times \mathbb R^{d \times N} \times \mathbb R^{d \times N} \to [0,+\infty)$, $\Phi: \Omega \times \mathbb R^d \times S^{N-1}\to [0, +\infty)$ 
such that  
\vspace{-0,25cm}
\begin{align*}
I_p(g,G) &=  \int_\Omega f(x,\nabla g(x),G(x)) \,  dx 
+ \int_{\Omega\cap S_g}\Phi(x,[g](x),\nu_g(x)) \, d\mathcal H^{N-1}(x),
\end{align*}
where the relaxed energy densities are given by
\begin{align*}
f(x_0, \xi, B) &:=	\limsup_{\varepsilon \to 0^+}
\frac{m(\xi(\cdot-x_0),B; Q(x_0,\varepsilon))}{\varepsilon^N},\\
\Phi(x_0,\lambda,\theta,\nu) & := \limsup_{\varepsilon \to 0^+}
\frac{m(u_{\lambda-\theta,\nu}(\cdot-x_0), 0; Q_{\nu}(x_0,\varepsilon))}
{\varepsilon ^{N-1}},
\end{align*}
for all $x_0\in \Omega$, $ \theta,\lambda \in \mathbb R^d$, 
$\xi, B \in \mathbb R^{d \times N}$ and $\nu \in \mathbb S^{N-1}$. In the above expressions, $0$ denotes the zero $\mathbb R^{d \times N}$ matrix,
$u_{\lambda-\theta, \nu}(y) := \begin{cases} 
\lambda - \theta, &\hbox{if } y\cdot \nu \geq 0\\
0, &\hbox{ if } y\cdot \nu < 0
\end{cases}$ and $m\colon SD(\Omega) \times \mathcal O_\infty(\Omega)\to [0,+\infty)$ is 
the Dirichlet-type functional in \eqref{Dirf} but defined with $L=1$ and 
$\mathcal F= I_p$.

If, in addition, $p>1$ and $\psi$ also satisfies $(\psi 3) - (\psi 5)$,  then $$\Phi(x_0,\lambda,\theta,\nu) = h_p(x_0,\lambda-\theta,\nu)$$
for every $x_0\in \Omega$, $\theta,\lambda \in \mathbb R^d$ and 
$\nu \in \mathbb S^{N-1}$, where $h_p$ is given in \eqref{hp}.
\end{theorem}

We consider the localised version of $I_p(g,G)$, defined, for every 
$O \in \mathcal O(\Omega)$, by
\begin{align}\label{Iploc}
I_p(g,G;O):= \inf\Big\{\liminf_{n\to\infty} E(u_n) &: u_n \in SBV(O;\mathbb R^{d}), u_n\debaixodaseta {L^1}{}  g, \,\,  \nabla u_n \weakst G \Big\},
\end{align}
where the above convergences take place in $O$ and the integrations in \eqref{Ex} are taken over $O$, instead of $\Omega$.
Then, the proof of the first part of Theorem \ref{measCF} is a consequence of Theorem \ref{BMZgm}, using also Remark \ref{inv}, and amounts to showing that $I_p(g,G;O)$ satisfies hypotheses $(H1)-(H4)$ stated therein. This yields an integral representation with relaxed densities expressed in terms of the functional $m$. Under the mentioned additional hypotheses, we prove further
that $\Phi(x_0,\lambda,\theta,\nu) = h_p(x_0,\lambda-\theta,\nu)$, by means of a double inequality.

As for the bulk energy density, we have (cf. \cite[Theorem 4.2]{BMZ2024}):

\begin{theorem}\label{Hphp}
Let $p \geq 1$. If, in addition to the conditions considered in the previous theorem, $W$ also satisfies (W$\,2$) and (W$\,3$) (if $p=1$),
then $$f(x_0,\xi, B) = H_p(x_0, \xi, B),$$
for every $x_0\in \Omega$, $\xi, B \in \mathbb R^{d \times N}$, where $H_p$ is given by \eqref{Hp}, and
$$\Phi(x_0,\lambda,\theta,\nu) = h_p(x_0,\lambda-\theta,\nu),$$
for every $x_0\in \Omega$, $\theta,\lambda \in \mathbb R^d$ and 
$\nu \in \mathbb S^{N-1}$, where $h_p$ is given in \eqref{hp}.
\end{theorem}

From Theorems \ref{measCF} and \ref{Hphp} we conclude that, assuming the same set of hypotheses as in \cite{CF1997} and \cite{MMOZ}, our relaxation approach via the global method, recovers the expressions for the relaxed energy densities 
$H_p$ and $h_p$ obtained in the cited papers, however Theorem \ref{BMZgm} is applicable under weaker assumptions.

\subsection{Iterated Relaxation}\label{iterel}

In this section, we present a further application of the global method approach for relaxation, now in the form of Corollary \ref{glm2}.

We begin by noting that, as was done in Section \ref{CFrel}, due to the approximation result in Theorem \ref{gapth}, see also Remark \ref{napprox},  we can associate to each multi-level structured deformation an energy satisfying hypotheses $(H1)$* - $(H4)$* of Corollary \ref{glm2}, through a relaxation process.

Stating it, for simplicity, for the case $L=2$, given 
$(g,G_1,G_2)\in SD_{2,p}(\Omega)$, we consider
\begin{align*}
I_{2,p}(g,G_1, G_2)&:= 
\inf_{\{u_{n_1,n_2}\}\subset SBV}\Big\{  \liminf_{n_1,n_2} E(u_{n_1,n_2}): u_{n_1,n_2}\weakstH(g,G_1,G_2)\Big\}, 
\end{align*}
where 
$E(\cdot)$ is given by \eqref{Ex}.
The fact that $I_{2,p}(g,G_1, G_2)$ satisfies $(H1)$* - $(H4)$* is proved in
\cite[Lemma 4.3, Lemma 4.4, Proposition 4.5]{BMMOZ2}, see also 
\cite[Theorem 4.6]{BMMOZ2}. Hence
Corollary \ref{glm2} provides an integral representation for the functional
$I_{2,p}(g,G_1, G_2)$, with the relaxed energy densities being expressed in terms of 
$m_{2,p}$ in \eqref{m2p}, taking $\mathcal F = I_{2,p}$.

We may also consider an iterated relaxation procedure in the following way.
We begin by relaxing the original energy in \eqref{Ex} in the manner of Choksi \& Fonseca and Matias, Morandotti, Owen \& Zappale, leading to the functional
$$I_{1}(g,G) := \inf\Big\{\liminf_{n\to\infty} E(u_n): u_n \in SBV, 
u_n \debaixodaseta{L^1}{} g, \nabla u_n \pweak G\Big\}.$$
As recalled in Section \ref{CFrel}, $I_{1}(g,G)$ admits an integral representation 
with relaxed densities $H_p$ and $h_p$ that, as shown in \cite[Theorem 4.2]{BMMOZ2}, still satisfy good properties that allows for the process to be iterated. 
Therefore, we can relax again and consider the functional
\begin{align*}
\widetilde{I}_{2,p} (g, G_1, G_2) & :=
\inf \Big\{  \liminf_{n_1}  I_{1}(\gamma_{n_1},\Gamma_{n_1}): \gamma_{n_1} \in SBV, \Gamma_{n_1}\in L^{p},\\
& \hspace{4,5cm}\gamma_{n_1}\debaixodaseta{L^1}{} g, 
\nabla\gamma_{n_1} \weakst G_1, \Gamma_{n_1}\weakst G_2
\Big\}.
\end{align*}
The existence of admissible sequences over which to perform the above minimisation is guaranteed by Theo\-rem \ref{approxthmCF} and Remark \ref{napprox}, for $\gamma_{n_1}$, whereas $\Gamma_{n_1}$ can be taken to be the constant sequence $G_2$.

In comparing the properties satisfied by the original densities $W$ and $\psi$ in \eqref{Ex}, with those that hold for $H_p$ and $h_p$, the main difference pertains to the growth conditions verified by $H_p$, when $p>1$. Indeed, due to the dependence of $H_p$ on both $W$ and $\psi$, the original growth condition of order $p>1$ is now replaced by a linear growth condition. This justifies the need for setting the global method for relaxation in the space $SD_{2,p}(\Omega)$ and for the version considered in Corollary \ref{glm2}.    

The two approaches to the relaxation for a three-level structured deformation, presented above, in fact, yield the same result since, as proven in \cite[Proposition 4.1]{BMMOZ2},
$$I_{2,p}(g,G_1, G_2) = \widetilde{I}_{2,p} (g, G_1, G_2).$$
The advantage of the direct strategy, via the global method, is that it can be applied, with the same proof, for any number of levels. The iterative strategy, on the other hand, provides more explicit formulae but requires proofs at each stage which become more and more involved as the number of levels increases (see \cite{BMMOZ2} for more details).

\bigskip

\subsection*{Acknowledgements}

The research of Ana Cristina Barroso was partially supported by National Funding from FCT - Funda\c c\~ao para a Ci\^encia e a Tecnologia through project 
UIDB/04561/2020: 

\noindent https://doi.org/10.54499/UIDB/04561/2020.

The research of José Matias was funded by FCT/Portugal through project UIDB/04459/2020 with DOI identifier 10-54499/UIDP/04459/2020.

Elvira Zappale acknowledges the support of Piano Nazionale di Ripresa e Resilienza (PNRR) - Missione 4 ``Istruzione e Ricerca''- Componente C2 Investimento 1.1, "Fondo per il Programma Nazionale di Ricerca e
Progetti di Rilevante Interesse Nazionale (PRIN)" - Decreto Direttoriale n. 104 del 2 febbraio 2022 - CUP 853D23009360006. She is a member of the Gruppo Nazionale per l'Analisi Matematica, la Probabilit\`a e le loro Applicazioni (GNAMPA) of the Istituto Nazionale di Alta Matematica ``F.~Severi'' (INdAM). 
She also acknowledges partial funding from the GNAMPA Project 2023 \emph{Prospettive nelle scienze dei materiali: modelli variazionali, analisi asintotica e omogeneizzazione}.
The work of EZ is also supported by Sapienza - University of Rome through the projects Progetti di ricerca medi, (2021), coordinator  S. Carillo e Progetti di ricerca piccoli,  (2022), coordinator E. Zappale.

\bibliographystyle{plain}

\end{document}